\documentclass[12pt,a4paper]{article}

\usepackage[english]{babel}
\usepackage[utf8]{inputenc}
\usepackage[T1]{fontenc}
\usepackage{amsmath, amsfonts, amssymb,amsthm,mathrsfs}
\usepackage{framed}
\usepackage{graphicx} 
\usepackage{array}
\usepackage{float}
\usepackage{subfig}
\usepackage[notindex,nottoc,notlot,notlof,notbib]{tocbibind}
\usepackage{makeidx}
\usepackage{hyperref}
\usepackage[table,usenames,dvipsnames]{xcolor}
\usepackage{multirow}
\usepackage{setspace}
\usepackage{empheq}
\usepackage{multicol}
\usepackage{fancyhdr}
\usepackage{sidecap}
\usepackage{wrapfig}
\usepackage{todonotes}
\usepackage{diagbox}
\usepackage{fourier}
\usepackage{enumitem}
\usepackage{verbatim}
\setlist[itemize]{label=$\bullet$}
\usepackage{listingsutf8}
\hypersetup{colorlinks = true, linkcolor = blue, citecolor = blue,
  urlcolor = blue}

\usepackage{titlesec}
\titleformat{\section}[block]
{\bfseries\Large}
{\thesection.}{5pt}
{\filcenter}
\titleformat{\subsection}[block]
{\bfseries\normalsize}
{\thesubsection.}{5pt}
{\filcenter}
\titleformat{\subsubsection}[block]
{\bfseries\normalsize}
{\thesubsubsection.}{5pt}
{\filcenter}
\theoremstyle{plain}
\newtheorem{theorem}{Theorem}[section]
\newtheorem{prop}[theorem]{Proposition}
\newtheorem{cor}[theorem]{Corollary}
\newtheorem{lem}[theorem]{Lemma}

\theoremstyle{remark}

\theoremstyle{definition}
\newtheorem{definition}[theorem]{Definition}

\newtheorem{rmk}[theorem]{Remark}

\DeclareMathOperator{\GL}{GL}
\DeclareMathOperator{\GU}{GU}

\DeclareMathOperator{\SL}{SL}

\DeclareMathOperator{\SU}{SU}

\DeclareMathOperator{\Aut}{Aut}
\DeclareMathOperator{\Out}{Out}

\DeclareMathOperator{\diag}{diag}

\DeclareMathOperator{\Image}{Im}

\DeclareMathOperator{\Res}{Res}

\DeclareMathOperator{\Irr}{Irr}
\DeclareMathOperator{\IBr}{IBr}

\DeclareMathOperator{\Uni}{Uni}
\DeclareMathOperator{\spec}{spec}

\newcommand{\leftexp}[2]{{\vphantom{#2}}^{#1}{#2}}

\def\F{\mathbb{F}}
\def\IG{\Irr(\tilde{G})}
\def\ILG{\IBr(\tilde{G})}
\def\tg{\tilde{G}}
\def\tbg{\tilde{\mathbf{G}}}
\def\ll{\left\langle}
\def\rr{\right\rangle}
\def\tc{\tilde{\chi}}
\def\te{\tilde{\eta}}

\def\me{\mathcal{E}}
\def\mo{\mathcal{O}}
\def\mos{\mathcal{O}^*}
\def\tvf{\tilde{\varphi}}

\def\mf{\mathcal{F}}
\def\mel{\mathcal{E}_{\ell}}

\def\ZO{Z(\tilde{G})\rtimes O(\tilde{G})}
\def\ZlpO{Z(\tilde{G})_{\ell'}\rtimes O(\tilde{G})}
\def\th{\Theta}
\def\tth{\tilde{\th}}
\def\tme{\tilde{\me}}
\def\w{\omega}
\def\fxi{\ll F\rr.\xi}
\def\ofq{\overline{\F}_q}

\begin{document}
\thispagestyle{empty}
\author{David Denoncin \footnote{Univ. Paris Diderot, Sorbonne Paris Cit\'e, CNRS, UMR 7586, Institut Math\'ematique de Jussieu - Paris Rive Gauche, 75205, Paris, France}}
\title{Stable basic sets for finite special linear and unitary groups}
\maketitle
\setlength{\parskip}{11pt}

\rule{\textwidth}{1pt}
\begin{abstract}
In this paper we show, using
Deligne-Lusztig theory and Kawanaka's theory of generalised
Gelfand-Graev representations, that the decomposition matrix of the
special linear and unitary group in non defining characteristic can be made
unitriangular with respect to a basic set that is stable under the action of automorphisms.
\end{abstract}
\rule{\textwidth}{1pt}

\newpage
\tableofcontents

\cleardoublepage\pagenumbering{arabic}
\setcounter{page}{1}
\renewcommand{\thetheorem}{\Alph{theorem}}
\setcounter{theorem}{0}

\section{Introduction}
\label{sec:introduction}

While ordinary characters of finite groups of Lie type are fairly well known, not so much is known about Brauer characters and in particular we do not have a general parameterisation for them. Decomposition matrices offer information about Brauer characters, and the unitriangularity of such matrices sets up a natural bijection from the set of Brauer characters to a corresponding so-called basic set of ordinary characters (see Definition \ref{def:unitri}). When we have such a bijection, we can try to obtain an equivariant one with respect to automorphisms. This is the case if and only if the basic set is stable under the action of automorphisms (see Lemma \ref{lem:equivar}).
 This is useful to deal with some counting conjectures (see \cite[Theorem 7.4]{CabSp13}). 

The unitriangularity of the decomposition matrices of finite general linear groups $\GL_n(q)$ in non defining characteristic $\ell$ was proved by Dipper in \cite{1Dip85} and \cite{2Dip85}, while the case of special linear groups $\SL_n(q)$ was done by Kleshchev and Tiep in \cite{KT}. However the techniques used in these papers rely on an actual construction of Brauer characters of $\GL_n(q)$. Such a construction is not known in the case of general unitary groups $\GU_n(q)$ so these methods cannot be applied for now to prove the unitriangularity result for either $\GU_n(q)$ or the special unitary groups $\SU_n(q)$. Nevertheless the result has been shown for $\GU_n(q)$ by Geck in \cite{Geck}, using Kawanaka's theory of Generalised Gelfand Graev Representations (GGGRs for short). Such a method can also be used to recover the result for $\GL_n(q)$. In \cite{GeckII} the same method was applied in the case of $\SU_n(q)$, but only the cases where $\ell \nmid \gcd(n, q+1)$ could be treated (see \cite[Theorem C]{GeckII}). In this case the usual basic set for $\GU_n(q)$ gives by restriction a basic set for $\SU_n(q)$. Investigating the methods developed in \cite{KT} and translating them in the context of Deligne-Lusztig theory and GGGRs, we first found that their methods could be adapted to  $\SU_n(q)$. But as was the case for $\SL_n(q)$, the many basic sets obtained for $\SU_n(q)$ were not stable with respect to automorphisms. In this paper we prove the following stronger statement:

\begin{theorem}\label{thm:A}
Let $q$ be a power of some prime number $p$ and $\ell$ be a prime number not dividing $q$. Let $n\geq 2$ and let $G\in \{\SL_n(q), \SU_n(q)\}$ be either the special linear or unitary group over a finite field with $q$ elements. Then $G$ has a unitriangular basic set in characteristic $\ell$ that is stable under the action of $\Aut(G)$.
\end{theorem}

Let $\tg=\GL_n(q)$ (resp.\ $\GU_n(q)$) and $G=\SL_n(q)$ (resp.\, $G=\SU_n(q)$). The unitriangular basic set that we obtain for $G$ is explicitly built from a unitriangular basic set for $\tg$. We develop in the first part, in a general setting, a condition that ensures the existence of a stable unitriangular basic set for $G$ provided that we already have one for $\tg$. The second part recalls some known facts about the character theory of general linear and unitary groups, and the last part is devoted to proving Theorem \ref{thm:A}.

I would like to thank Marc Cabanes for his support and suggestions of improvements regarding the manuscript. I am grateful to Gunter Malle  for inviting me to work on this project at T.U. Kaiserslautern with support of ERC Advanced Grant 291512. I also thank Olivier Brunat and Jay Taylor for fruitful discussions.

\renewcommand{\thetheorem}{\thesection.\arabic{theorem}}

\section{Stable unitriangular basic set for a normal subgroup}
\label{sec:unitr-decomp-matr}
\label{sec:unitrigeneral}

\subsection{Decomposition numbers, unitriangularity and stability under group action}
\label{sec:decomp-numb-unions}

Our basic reference for the representation theory of finite groups is \cite{NT}. Let $H$ be a finite group and $\ell$ be a prime number dividing the order of $H$. The set of ordinary irreducible characters of $H$ will be denoted by $\Irr(H)$ and its elements will be referred to simply as irreducible characters. The set of (irreducible $\ell$-) Brauer characters of $H$ will be denoted by $\IBr_{\ell}(H)$ or simply $\IBr(H)$. The set $\Irr(H)$ forms a basis of the space of class functions on $H$, which is endowed with its usual scalar product denoted by $\ll -,-  \rr$. The set $\IBr(H)$ is a basis of the subspace of class functions on $H$ vanishing outside the set of elements of order coprime to $\ell$ (also called $\ell'$-elements). The set of $\ell'$-elements of $H$ will be denoted by $H_{\ell'}$. We will denote by $d^1: \mathbb{Z}\Irr(H)\rightarrow \mathbb{Z}\IBr(H)$ the linear map consisting in multiplying  by the characteristic function of $H_{\ell'}$ (see \cite[Definition 5.7]{CabEn}), also called the \emph{decomposition map} (for $H$ is characteristic $\ell$). If $\chi\in\Irr (H)$, we have $d^1(\chi)\in{\mathbb{N}}\IBr(H)$ and we write (see \cite[Chapter 3, \S 6 Equation 6.2]{NT})

\begin{equation}
  d^1(\chi)=\sum_{\eta\in\IBr(H)}{d_{\chi,\eta}\eta}.
\end{equation}
The integers $d_{\chi,\eta}$ are the \emph{decomposition numbers}. Put together they form the \emph{decomposition matrix} of $H$
.

\begin{definition}\label{def:mel}
If $\me\subseteq\Irr (H)$, we denote by $\me_\ell\subseteq \IBr(H)$ 
the set of irreducible components of all $d^1(\chi)$'s for $\chi\in\me$, \textsl{i.e.}
\begin{displaymath}
  \mel:=\{\eta\in\IBr(H)\mid\exists \ \chi\in\me, \ d_{\chi,\eta}\neq 0\}.
\end{displaymath}
\end{definition}

We make the following definition which allows us to deal more easily with decomposition matrices (see \cite[Theorem 1.4]{KT}). It is equivalent to the existence of a square submatrix of the decomposition matrix being actually unitriangular.

\begin{definition}\label{def:unitri}
The group $H$ is said to have a \emph{unitriangular decomposition matrix} (in characteristic $\ell$) if there is a 
partial order relation $\leq$ on $\IBr(H)$ and an injective map 
$\th : \IBr(H)\rightarrow \Irr(H)$ such that the decomposition numbers $d_{\th 
(\eta ),\eta '}$ for $\eta$, $\eta '\in\IBr(H)$ are 0 unless $ \eta '\leq 
\eta $ and are 1 whenever $ \eta = \eta '$. (Note that any map $\th$ 
satisfying those last two conditions has to be injective.)
\end{definition}
 If this is the case then the subset $\th(\IBr(H))$ of $\Irr(H)$ will be called a \emph{unitriangular basic set} (in characteristic $\ell$) and,
in particular, for any $ \eta\in \IBr(H)$ we have
\begin{equation}\label{eq:red}
d^1(\th( \eta )) =  \eta +\sum_ { \eta '\in\IBr(H);\  \eta '\lneq 
\eta }d_{\th( \eta ) , \eta '} \eta ',
\end{equation}

where $\eta'\lneq \eta$ means $\eta'\leq \eta$ and $\eta'\neq \eta$.

Note also that if $H$ has a unitriangular decomposition matrix with respect to
$(\leq ,\th )$, it also has a unitriangular decomposition matrix with respect to any pair $(\leq ',\th )$ where $\leq 
'$ is any order relation implied by $\leq$. In particular, one has the 
property for any linear order 
refining our original order.

We prove the following lemma which will simplify checking that the map $\tth$ provided by Theorem \ref{thmMeinolf} is equivariant. It also allows one to talk about stable basic sets without referring to the equivariant map it is the image of, thus shortening statements (see Theorem \ref{thm:equiv}).

\begin{lem}\label{lem:equivar}
  Assume that the decomposition matrix of $H$ is unitriangular with respect to a pair $(\leq,\th)$. Let $K$ be a finite group acting on both the sets $\IBr(H)$ and $\Irr(H)$ and suppose that $d^1$ is $K$-equivariant (the action of the group $K$ being linearly extended). Let $\me=\th(\IBr(H))$. Then the map $\th$ is $K$-equivariant if and only if the set $\me$ is $K$-stable (\textit{i.e.}, for all $\chi\in\me,k\in K \quad k.\chi\in\me$). 
\end{lem}
\begin{proof}
The "only if" part of the statement is obvious. For the rest of this proof we set, for any $\eta\in\IBr(H)$,
  \begin{displaymath}
    \mathcal{I}_{\eta}:=\{\eta'\in\IBr(H)\mid\eta'\leq\eta\}.
  \end{displaymath}
Let $\eta\in\IBr(H)$ and $k\in K$. The $K$-equivariance of the map $d^1$ and the $K$-stability of $\me$ shows that there exists $\eta'\in\mathcal{I}_{\eta}$ such that
  \begin{displaymath} 
    \Theta^{-1}(k.\Theta(\eta))=k.\eta'.
  \end{displaymath}
Hence the result follows by induction on the cardinality of the set $\mathcal{I}_{\eta}$ using the injectivity of the map $\Theta$.
\end{proof}

\subsection{Unitriangularity and normal inclusion}
\label{sec:2.2}
 In what follows $G\lhd \tilde G$ is a normal inclusion of finite groups with cyclic 
factor group $\tilde G/G$, and $\ell$ is a prime number dividing the order of $G$. Recall that by Clifford theory, restriction to $G$ of an irreducible ordinary or Brauer character of $\tg$ is multiplicity-free (see \cite[Chapter 3, Theorems 3.1 and 5.7]{NT}). Denote by $A$ the abelian group of linear characters 
of $\tg$ that are trivial on $G$. It acts by multiplication on both the sets
$\IG$ and $\ILG$; the orbits coincide with the 
fibers of the map $\Res_G$ (see \cite[Lemma 3.7]{KT}). Elements in a finite group $H$ whose order is a power of $\ell$ will be called $\ell$-elements and the corresponding set is denoted by $H_{\ell}$. Note that the group $A_{\ell}$ acts trivially on $\ILG$. We also let $\tg$ act on both 
$\Irr(G)$ and $\IBr(G)$ by conjugation.

For $\tc\in\IG$ (resp.\  
$\te\in\ILG$), we denote by $\Irr (G\mid\tc)$ (resp.\ 
$\IBr(G\mid\te)$) the set of irreducible components of its 
restriction to $G$. Following \cite[\S 3.1]{KT}, we make the following definition.

\begin{definition}
  For $\tc\in\IG$ (resp.\ $\te\in\ILG$), we let $\kappa_G(\tc)$ (resp.\ $\kappa_G(\te)$) denote the cardinality of the set $\Irr(G \mid \tc)$ (resp.\ $\IBr(G \mid \te)$).
\end{definition}

We now have the following to obtain a unitriangular basic set for $G$, assuming $\tg$ has one.
\begin{theorem}\label{thm:1}
Assume $\tg$ has a unitriangular decomposition matrix for the pair $(\leq ,\tth )$. Moreover, assume that  for 
any $\te\in \ILG$,  
\begin{displaymath}
\kappa_G(\te) = \kappa_G(\th(\te)) .\leqno({\rm H})
\end{displaymath}
 
Then $G$ has a unitriangular decomposition matrix for a pair $(\leq_G,\th )$ which can 
be chosen as follows.

Assuming $\leq$ is a linear order (or denoting by the same symbol any 
linear order refining our initial $\leq$), let $\tilde{R}$ be the 
representative system of $\ILG$ mod the action of $ A_{\ell'}$ where one 
selects in each class the $\leq$-minimal element. If $\eta$, $\eta '\in 
\IBr(G)$, we say $\eta\leq_G\eta '$ if and only if 
$\eta =\eta '$ or there are $\te$, $\te '\in \tilde{R}$ with $\te \lneq\te '$ and $\eta\in \IBr (G\mid \te)$, $\eta '\in 
\IBr (G\mid \te')$.

Moreover $\th: \IBr(G)\rightarrow \Irr (G)$ can be chosen to be $\tg$-equivariant.
\end{theorem}

\begin{rmk} \label{rmk:multiple}
Note that without the assumption (H), the 
unitriangularity of the decomposition matrix of $\tg$ would 
imply that $\kappa_G(\te) $ is a \textbf{multiple} of $\kappa_G(\tth(\te))$ (see the proof below).
\end{rmk}

\begin{proof}
We can assume $\leq$ to be a linear 
order. It is clear that the relation $\leq_G$ defined in the statement of the theorem is 
a partial order. 

If $\te\in\tilde{R}$, one has
  from Equation \ref{eq:red}
\begin{displaymath} 
d^1(\tth(\te )) = \te +\sum_ {\te '\in\ILG ; \ \te 
'\lneq\te }d_{\tth(\te ) ,\te '}\te '.
\end{displaymath}
Using Clifford theory, we get by restricting to $G$
\begin{equation}\label{eq:proof}
  \sum_{\chi\in \Irr (G\mid\tth (\te )) }d^1(\chi) = 
\sum_{\eta\in \IBr (G\mid \te) }\eta +\sum_{\te 
'\in\ILG    ;\ \te '\lneq\te }d_{\tth(\te) ,\te 
'}\Res_G \te '  .
\end{equation}
In Equation \ref{eq:proof}, a given $\eta\in \IBr (G\mid \te) $ is present only 
once in the whole right hand side since it cannot be present in the last 
sum $\sum_ {\te '\in\ILG;\ \te '\lneq\te 
}d_{\tth(\te ) ,\te '}\Res_G \te '  $. Indeed the latter 
would imply that $\eta$ is present in some $\Res_G\te '$ with 
$\te ' \lneq \te$. Since $\eta$ is also in $\Res_G\te$, 
this implies that $\te$ and $\te'$ are in the same $A_{\ell'}$-orbit. 
By the definition of $\tilde{R}$, we then have $\te\leq\te '$, a 
contradiction. We conclude that for any $\eta\in\IBr(G\mid\te)$, there exists a unique $\chi\in\Irr(G\mid\tth(\te))$ such that $d_{\chi,\eta}\neq 0$ or equivalently $d_{\chi,\eta}=1$.

The above assignement $\eta\mapsto\chi$ defines a map $\th 
\colon\IBr (G\mid \te) \to  \Irr (G\mid\tth (\te))$ and 
which is $\tg$-equivariant since $d_{\chi ,\eta}=d_{\leftexp{\sigma}{\chi} 
,\leftexp{\sigma}{\eta}}$ for any $\sigma\in \Aut(G)$. Note that by Clifford 
theory both $\IBr(G\mid \te) $ and $  \Irr (G\mid\tth (\te 
))$ are a $\tg$-orbit and thus the map $\th$ is surjective. The existence of a $\tg$-equivariant map $\th$ 
then implies the divisibility of Remark \ref{rmk:multiple}. Moreover (H) implies 
that $\th$ maps $\IBr (G\mid \te)$ bijectively on $\Irr (G\mid\tth 
(\te))$.

The various sets $\IBr (G\mid\te)$ for $\te$ ranging over 
$\tilde{R}$ partition the set $\ILG$. So we define $\th$ as above 
independently on each set.

It remains to show that for any $\eta $, $\eta ''\in\ILG$ we have 
$d_{ \th (\eta  ) ,\eta ''} = 0$ unless $\eta ''\leq_G\eta$.
Assume we have $d_{ \th (\eta ) ,\eta ''} \not= 0$. By construction of the map $\th$ we can assume that $\eta ''$ is present in
the last sum in Equation \ref{eq:proof}, more precisely that it is present in some $\Res_G\te '$ with $\te 
'\lneq\te$. By the definition of $\tilde{R}$, letting $\{ \te ''\} =\tilde{R}\cap A_{\ell'}.\te '$, we then have $\te ''\leq \te '\lneq\te$, 
and therefore $\te ''\lneq\te$. Since $\eta ''\in \IBr (G\mid 
\te')=\IBr (G\mid \te'')$, we get $\eta ''\leq_G\eta$ by our 
definition of $\leq_G$.
\end{proof}

\begin{rmk}
  Note that the basic set obtained for $G$ is not explicit as the image of the map $\th$ is the set $\cup_{\te\in\tilde{R}}{\Irr(G\mid\tth(\te))}$ and the set $\tilde{R}$ is not known. However if the map $\tth$ is $A_{\ell'}$-equivariant then the equality $\Irr(G\mid\tth(\te))=\Irr(G\mid\tth(\te'))$ for any $\te\in\ILG$ and $\te'\in A_{\ell'}.\te$ would imply that $\Image(\th)=\cup_{\te\in\IBr(\tg)}{\Irr(G\mid\tth(\te))}=\cup_{\tc\in\tth(\ILG)}\Irr(G\mid\tc)$. Thus in this case the unitriangular basic set for $G$ is explicitly constructed from that of $\tg$. For this reason in what follows we will suppose that the map $\tth$ is $A_{\ell'}$-equivariant.
\end{rmk}

\subsection{Unitriangularity, normal inclusion and stability}

We continue with a normal inclusion $G\triangleleft \tg$ such that $\tg/G$ is a cyclic group and $\ell$ a prime number dividing the order of $G$. We will now give an adapted setup to prove Theorem \ref{thm:A} in the framework provided by \cite{Geck}; given a unitriangular basic set and a partial pre-order on $\IG$, we prove an equivariant version of Theorem \ref{thm:1} (see Theorem \ref{thm:equiv}) that will be used to prove Theorem \ref{thm:A}. Finally, to prove Theorem \ref{thm:A} we will have to modify a given basic set and we first introduce some terminology to be able to focus on characters instead of equivariant maps.

Suppose that we are given a partial pre-order $\leq$ on the set $\IG$ and an injective map $\tth : \IBr(\tg) \rightarrow \Irr(\tg)$, such that $\leq$ is an actual  partial order on $\tth(\IBr(\tg))$. Then we naturally obtain a partial order on the set $\IBr(\tg)$ that we will still denote by $\leq$. Suppose that the pair $(\leq, \tth)$ makes the decomposition matrix of $\tg$ unitriangular and let $\tme=\tth(\ILG)$ be the corresponding unitriangular basic set.

\begin{definition}\label{def:exchange}
    Let $\tc\in\tme, \tc'\in\Irr(\tg)$. We say that $\tc$ is \textsl{replaceable} by $\tc'$ if:
    \begin{itemize}
      \item there exists a unique $\leq$-maximal element in $\{\tilde{\eta}\in\IBr(\tg) \mid d_{\tc', \tilde{\eta}}\neq 0\}$,
      \item that maximal element is $\tth^{-1}(\tc)$, and $d_{\tc', \tth^{-1}(\tc)}=1$.
    \end{itemize}
If moreover we have the equality $\kappa_G(\tc')=\kappa_G(\tth^{-1}(\tc))$, we will say that $\tc$ is \textsl{$\kappa_G$-replaceable} by $\tc'$ (this relation is not symmetric).
\end{definition}

\begin{rmk}\label{rmk:exchange}
    Note that if $\tc$ is replaceable by $\tc'$ then the map $\tth':\IBr(\tg)\rightarrow\Irr(\tg), \te\mapsto \tth(\te)$ if $\te\neq\tth^{-1}(\tc)$ and $\tth'(\tth^{-1}(\tc))=\tc'$ also makes the decomposition matrix of $\tg$ unitriangular for the order $\leq$ on the set $\IBr(\tg)$. This corresponds to exchanging the two rows of the decomposition matrix indexed by $\tc$ and $\tc'$, such that this change preserves unitriangularity. Note also that if $\tc$ is replaceable by $\tc'$ then $\kappa_G(\tc')$ divides $\kappa_G(\tth^{-1}(\tc))$ by Remark \ref{rmk:multiple}. This provides a useful tool to get a lower bound on $\kappa_G(\tth^{-1}(\tc))$ (see Proposition \ref{lemineq}).
\end{rmk}

With those definitions Theorem \ref{thm:1} can be rewritten as follows:
\begin{theorem}
  Assume that $\tg$ has a unitriangular basic set and that all its elements are $\kappa_G$-replaceable. Then $G$ has a unitriangular basic set.
\end{theorem}

We now introduce an equivariant version of this theorem that we will use to prove Theorem \ref{thm:A}.

Let $O$ be a subgroup of $\Aut(\tg)_G$ (where $\Aut(\tg)_G$ is the subgroup of $\Aut(\tg)$ stabilising $G$). It acts on any of the sets $\ILG, \IG, \Irr(G)$ and $\IBr(G)$ by the formula:
\begin{equation}
\label{eq:actionofaut}
  \forall F_0\in O,\quad F_0(\varphi)= \varphi \circ F_0^{-1}.
\end{equation}

Recall that we denote by $A$ the group of linear characters of $\tg$ that are trivial on $G$. For any $\tc\in\IG$, denote by $(A_{\ell'}\rtimes O)_{\tc}$ the group stabiliser of $\tc$ in $A_{\ell'}\rtimes O$. It is clear that the map $d^1$ is $A_{\ell'}\rtimes O$-equivariant. 
\begin{definition}\label{def:Oexchange}
    Let $\tc\in\tme, \tc'\in\Irr(\tg)$. We say that $\tc$ is \textsl{$O$-replaceable} by $\tc'$ if:
    \begin{itemize}
      \item $\tc$ is $\kappa_G$-replaceable by $\tc'$, and
      \item $(A_{\ell'}\rtimes O)_{\tc}=(A_{\ell'}\rtimes O)_{\tc'}$.
    \end{itemize}
\end{definition}

From now on and until the end of this section we will assume that $\tme$ is $A_{\ell'}\rtimes O$-stable, and this next lemma is clear using the equivariance of the map $\tth$ (see Lemma \ref{lem:equivar}).

\begin{lem}\label{lem:equivex}
    Let $\tc\in\tme, \tc'\in\Irr(\tg)$ and assume that $\tc$ is $O$-replaceable by $\tc'$. Then for all $\hat{z}\in A_{\ell'},F_0\in O$, $\hat{z}F_0(\tc)$ is $\kappa_G$-replaceable by $\hat{z}F_0(\tc')$  .
\end{lem}

The assumption of $O$-replaceability allows one to exchange characters in an $A_{\ell'}\rtimes O$-compatible way in the sense of this next lemma.
\begin{lem}\label{lem:equivfinal}
    Assume that for all $\tc\in\tme$, there exists $\tc'\in\Irr(\tg)$ such that $\tc$ is $O$-replaceable by $\tc'$. Then there exists a subset $\tme'\subseteq\Irr(\tg)$ and an $A_{\ell'}\rtimes O$-equivariant bijection $\Delta : \tme \rightarrow \tme'$ such that for all $\tc\in\tme$, $\tc$  is  $\kappa_G$-replaceable  by $\Delta(\tc)$.
\end{lem}
\begin{proof}
    Let $\hat{\me}$ be a representative system of $\tme$ mod the action of $A_{\ell'}\rtimes O$. For each $\tc\in\hat{\me}$ let $\Delta(\tc)=\tc'$. Then by hypothesis, for each $\tc\in\hat{\me}$ the $A_{\ell'}\rtimes O$-orbits of $\tc$ and $\Delta(\tc)$ are isomorphic as $A_{\ell'}\rtimes O$-sets. Thus we can extend $\Delta$ on $\tme$ into an $A_{\ell'}\rtimes O$-equivariant map by setting $\Delta(\hat{z}F_0(\tc))=\hat{z}F_0(\Delta(\tc))$ for any $\tc\in\hat{\me}$ and any $\hat{z}\in A_{\ell'},F_0\in O$. Now Lemma \ref{lem:equivex} gives the result.
\end{proof}

We synthesise all this subsection into a criterion on a unitriangular basic set of $\tg$ to obtain an $O$-stable unitriangular basic set for $G$.
 
\begin{theorem}\label{thm:equiv}
  Assume that $\tg/G$ is a cyclic group and let $O$ be a subgroup of $\Aut(\tg)_G$. Assume that $\tg$ has an $A_{\ell'}\rtimes O$-stable unitriangular basic set such that all its elements are $O$-replaceable. Then $G$ has a unitriangular basic set which can be chosen to be stable under the actions of $O$ and $\tg$.
\end{theorem}
\begin{proof}
Let $(\leq, \tth)$ be the partial order and map making the decomposition matrix of $\tg$ unitriangular and let $\tme=\tth(\IBr(\tg))$ be the corresponding unitriangular basic set. Let $\Delta$ be a map constructed as in Lemma \ref{lem:equivfinal}. The map $\tth'=\Delta\circ \tth$ then satisfies the hypothesis of Theorem \ref{thm:1}. Let $\Theta'$ be a $\tg$-equivariant map obtained as a result of the proof of Theorem \ref{thm:1} whose notation we adapt for the rest of this proof. It remains only to prove that $\Theta'$ is $O$-equivariant, as the action of $\tg$ commutes with the decomposition map of $G$ (see Lemma \ref{lem:equivar}). Let $\eta\in\IBr(G)$, $\te\in \tilde{R}$ and $F_0\in O$. The character $F_0(\te)$ may not be in the set $\tilde{R}$ but there exists $\hat{z}\in A_{\ell'}$ such that $\hat{z}F_0(\te)\in\tilde{R}$. Then by construction (see the proof of Theorem \ref{thm:1}) we have  $\th'(\eta)\in\Irr(G\mid\tth'(\te))$  and $d_{\th'(\eta),\eta}=1$.
This implies that $F_0(\th'(\eta))\in\Irr(G\mid F_0(\tth'(\te)))$ and $d_{F_0(\th'(\eta)),F_0(\eta)}=1$.
By equivariance of $\tth'$ we have 
\begin{displaymath}
\Irr(G\mid F_0(\tth'(\te)))=\Irr(G\mid \hat{z}F_0(\tth'(\te)))=\Irr(G\mid \tth'(\hat{z}F_0(\te))).  
\end{displaymath}
Finally the fact that $\Theta'(F_0(\eta))$ is the unique element $\chi$ of $\Irr(G)$ such that $\chi\in\Irr(G\mid \tth'(\hat{z}F_0(\te)))$ and $d_{\chi, F_0(\eta)}=1$
implies that $F_0(\Theta'(\eta))=\Theta'(F_0(\eta))$, and concludes the proof.
\end{proof}

\section{Reminders of the character theory of the finite general linear and unitary groups}

In this section we recall some fairly well known material about the character theory of finite general linear and unitary groups (see \cite[pp. 111--112]{FS}, \cite[pp. 149--152]{DM}, \cite[Chapter 8]{CabEn} or \cite[\S 3]{Br}). We do so in a setup that will ease the proof of Theorem \ref{thm:A} and the notation introduced here will be in force until the end of this paper.

\textbf{Notation}. Let $n\geq 2$ be an integer and let $\overline{\F}_q$ be an algebraic closure of a field
$\mathbb{F}_{q}$ with $q$ elements, where $q$ is a power of a prime number $p$. Let $\ell$ be a prime number not dividing $q$.

Let $w_0$ be the permutation matrix in $\GL_n(\ofq)$ corresponding to the element $(1\ n)(2 \ n-1)\dots$ of the symmetric group on $n$ elements, the latter being denoted by $\mathfrak{S}_n$. The Frobenius map $F : \GL_n(\ofq)\rightarrow\GL_n(\ofq)$ defined by $F((a_{i,j})):=(a_{i,j}^q)$ (resp.\ $F((a_{i,j})):=w_0\leftexp{t}{(a_{i,j}^q)^{-1}}w_0$) will be called \emph{untwisted} (resp.\, \emph{twisted}). If $F$ is twisted then the fixed point group $\GL_n(\ofq)^F$ is the finite general unitary group denoted by $\GU_n(q)$ and it may also be denoted by $\GL_n(-q)$, and the fixed point group $\SL_n(\ofq)^{F}$ is the finite special unitary group denoted by $\SU_{n}(q)$. In the untwisted case we use the standard notation $\GL_n(\ofq)^F=\GL_n(q)$ and $\SL_n(\ofq)^F=\SL_n(q)$. 

Moreover we will use standard notation from Deligne-Lusztig theory (see \cite{CabEn}). For instance if $\mathbf{G}$ is a connected reductive algebraic group endowed with a rational structure, and if $\mathbf{L}$ is a rational Levi subgroup of $\mathbf{G}$, then $R_{\mathbf{L}}^{\mathbf{G}}$ will denote the Deligne-Lusztig induction map. If $t$ is a semisimple element in $\GL_n(\ofq)^F$, then $\me(\GL_n(\ofq)^F,t)$ will denote the Lusztig series associated with the $\GL_n(\ofq)^F$-conjugacy class of $t$ (see \cite[\S 8.4]{CabEn}, note that $\GL_n(\ofq)$ is self-dual).

From now on and until the end of this paper, $n$, $q$ and $\ell$ are fixed.

We set $\tbg:=\GL_n(\ofq)$, $\mathbf{G}:=\SL_n(\ofq)$, we let $F$ be the twisted or untwisted Frobenius endomorphism of $\tbg$ and we set $\tg:=\tbg^F$ and $G:=\mathbf{G}^F$. Furthermore we fix $\mathbf{T}$ the $F$-stable maximal torus consisting of diagonal matrices contained in the $F$-stable Borel subgroup of upper triangular matrices of $\tbg$, and we denote by $\mathbf{W}=N_{\tbg}(\mathbf{T})/\mathbf{T}$ the Weyl group of $\tbg$. It is isomorphic to the symmetric group $\mathfrak{S}_n$.

\subsection{Jordan decomposition}
\label{sec:JordanDec}

We let the Frobenius endomorphism $F$ act on $\overline{\F}^{\times}_q=\GL_1(\overline{\F}_q)$ by setting for all $x\in\overline{\F}^{\times}_q$, $F(x):=x^{q}$ (resp.\ $F(x):=x^{-q}$) in the untwisted (resp.\ twisted) case.

For $t\in \tg$ a semisimple element, we denote by:
\begin{itemize}
  \item $\tbg_t$ the centraliser $C_{\tbg}(t)$ of $t$ in $\tbg$, and by $\tg_t$ the $F$-fixed point group $\tbg_t^F$,
  \item $\spec(t)$ the set of eigenvalues of $t$ as an element of $\tbg$,
  \item $m_{\xi}(t)$ the dimension of the $\xi$-eigenspace of $t$ as an element of $\tbg$ for $\xi\in \spec(t)$,
  \item $\ll F\rr.\xi:=\{F^{i}(\xi)\mid i \geq 0\}$ the $F$-orbit of $\xi\in\spec(t)$, and
  \item $\spec(t)/F:=\{\ll F \rr. \xi\mid \xi\in \spec(t)\}$ the set of $F$-orbits in $\spec(t)$.
\end{itemize}

Recall (see for example \cite[p.\ 111]{FS}) that the set of semisimple conjugacy classes in $\tg$ is in bijection with the set of conjugacy classes in $\tbg$ of semisimple elements $t\in\tbg$ satisfying the following conditions:
\begin{itemize}
  \item $\spec(t)$ is $F$-stable, and
  \item the map $\spec(t)\rightarrow \mathbb{Z}_{\geq 0}, \xi \mapsto m_{\xi}(t)$ is constant on $F$-orbits.
\end{itemize}

We make the following definition for dealing efficiently with multipartition indexed by $F$-orbits of eigenvalues.

\begin{definition}\label{def:mft}
  For $t$ a semisimple element of $\tg$ we denote by $\mf_t$ the set of maps $\lambda$ defined from the set $\spec(t)/F$ to the set of all partitions of non-negative integers such that
  \begin{displaymath}
    \text{for all } \fxi \in \spec(t)/F, \quad \lambda(\fxi) \text{ is a partition of the integer } m_{\xi}(t).
  \end{displaymath}
\end{definition}

Let $t$ be a semisimple element in $\tg$. Then it is contained in some $F$-stable maximal torus $\mathbf{T}'$ of $\tbg$ and we let $g\in\tbg$ such that $\leftexp{g}{\mathbf{T}}=\mathbf{T}'$. Let us also set $\mathbf{t}:={t}^g\in\mathbf{T}$. We let 
\begin{displaymath}
\mathbf{W}_{\mathbf{t}}:=\{\sigma\in \mathbf{W} \mid \mathbf{t}^{\sigma}=\mathbf{t}\},  
\end{displaymath}
and set $\mathbf{W}_t:=\leftexp{g}{W_{\mathbf{t}}}$ the \emph{relative Weyl group of $t$}. If $\dot{w}:=g^{-1}F(g)$ and $w$ is the canonical image of $\dot{w}$ in $\mathbf{W}$ then conjugation with $g$ induces an isomorphism of groups $\tbg_{\mathbf{t}}^{wF}\cong \tbg_t^F$ where $\tbg_{\mathbf{t}}:=C_{\tbg}(\mathbf{t})$ and $wF$ is $F$ followed by conjugation with $\dot{w}$. We have 
\begin{equation}\label{eq:centraliser}
{\tbg_{\mathbf{t}}}=\prod_{\xi\in \spec(t)}{\GL_{m_{\xi}(t)}(\overline{\F}_q)},
\end{equation}

and $wF$ acts on $\tbg_{\mathbf{t}}$ in a such a way that
\begin{equation}\label{eq:actiononG}
    wF(\GL_{m_{\xi}(t)}(\overline{\F}_q))=\GL_{m_{F(\xi)}(t)}(\overline{\F}_q).
\end{equation}

We also have
\begin{equation}\label{eq:weylgroup}
    \mathbf{W}_{\mathbf{t}}=\prod_{\xi\in\spec(t)}\mathfrak{S}_{m_{\xi}(t)},
\end{equation}
and $wF$ acts on $\mathbf{W}_{\mathbf{t}}$ in such a way that 
\begin{equation}\label{eq:actiononWG}
wF(\mathfrak{S}_{m_{\xi}(t)})=\mathfrak{S}_{m_{F(\xi)}(t)}.    
\end{equation}

Partitions label irreducible characters of symmetric groups and we will use the parameterisation of \cite[Theorem 2.3.15]{JK} \textbf{twisted by conjugation of partitions}, so that the partition $(1, \dots, 1)$ will label the trivial character of $\mathfrak{S}_n$ and the partition $(n)$ will label the sign character. Then Equation \ref{eq:actiononWG} establishes naturally a bijection $\lambda\mapsto\rho_{\lambda}$ between the sets $\mf_{t}$ and $\Irr(\mathbf{W}_{t})^{F}$. 

By \cite[Theorem 2.2]{LS}, the set   $\me(\tg_{t},1)$ is in bijection with the set $\Irr(\mathbf{W}_{t})^{F}$ and the bijection is as follows. For $\rho\in\Irr(\mathbf{W}_{t})^{F}$, let $\tilde{\rho}$ be one of the two extentions of $\rho$ to the group $ \mathbf{W}_{t}\rtimes \ll F\rr$ that is realisable over the rational field (see \cite[Proposition 3.2]{Lus84}). Then there exists a sign $\epsilon_{\rho}$ such that 
\begin{equation}
    \tvf:=\epsilon_{\rho}\frac{1}{|\mathbf{W}_{t}|}\sum_{\sigma\in \mathbf{W}_{t}}\tilde{\rho}(\sigma.F)R_{{{\mathbf{T}}}_{\sigma}}^{\tbg_{t}}(1)
\end{equation}
is a unipotent character of $\tg_{t}$, where $\mathbf{T}_{{\sigma}}:=\leftexp{gh}{\mathbf{T}}$ is an $F$-stable maximal torus of $\tbg_{t}$ for some $h\in\tbg_{\mathbf{t}}$ with $h^{-1}F(h)$ having image $\leftexp{g^{-1}}{\sigma}$ in $\mathbf{W}_{\mathbf{t}}$.

\begin{definition}\label{def:mos1}
For $\lambda\in\mf_t$ we denote by $\tvf_{\lambda}$ the element in $\me(\tg_t,1)$ such that
\begin{equation}\label{eq:unipotentgt}
    \tvf_{\lambda}=\epsilon_{\rho_{\lambda}}\frac{1}{|\mathbf{W}_{t}|}\sum_{\sigma\in \mathbf{W}_{t}}\tilde{\rho}_{\lambda}(\sigma.F)R_{{{\mathbf{T}}}_{\sigma}}^{\tbg_{t}}(1).
\end{equation}
This assignment $\lambda\mapsto \tvf_{\lambda}$ is thus a bijection between $\mf_t$ and $\me(\tg_t,1)$.

\end{definition}
We will naturally label a unipotent conjugacy class in $\tbg_t$ by multipartitions via the Jordan normal form. Equation \ref{eq:actiononG} establishes naturally a bijection $\lambda\mapsto \mo_{\lambda}$ between the set $\mf_t$ and the set $\Uni(\tbg_t)^F$ of $F$-stable unipotent conjugacy classes of $\tbg_t$. Note also that $\Uni(\tbg_t^F)=\Uni(\tbg_t)^F$.

\begin{rmk}\label{rmk:springercorr}
 The assignement $\rho_{\lambda}\otimes\epsilon\mapsto \mo_{\lambda}$ is the \emph{Springer correspondence} as in \cite{LS85} (see \cite[Proposition 5.2]{LS85}), where $\epsilon$ is the sign character.
\end{rmk}

The Jordan decomposition of characters gives us a bijection between the set $\mf_t$ and the set $\me(\tg,t)$ that we denote by $\lambda \mapsto \chi^{\tg}_{t,\lambda}$. This bijection is explicit in the following way. Recall that there exists an isomorphism between the group $Z(\tg_t)$ and the group of linear characters of $\tg_t$, that we denote by $\hat{}$ (see \cite[Equation 8.19]{CabEn}). Then $\tc^{\tg}_{t,\lambda}=\epsilon_{\tbg}\epsilon_{\tbg_t}R_{\tbg_{t}}^{\tbg}(\hat{t}\tvf_{\lambda})$ (see \cite[Theorem 15.10]{CabEn}) where $\epsilon_{\tbg}$ and $\epsilon_{\tbg_t}$ are well defined signs (see \cite[ p.\ 126]{CabEn})

\subsection{Equivariance of the Jordan decomposition}



\begin{definition}
\label{def:OG}
We let $O(\tg)$ be the subgroup of the group of automorphisms of the abstract group $\tbg$ generated by the elements $F_p:(a_{i,j})\mapsto (a_{i,j}^p)$ and $\gamma_0:(a_{i,j})\mapsto w_0\leftexp{t}{(a_{i,j})^{-1}}w_0$. 
\end{definition}
Note that if $q=p^m$ then we have $F=\gamma_0F_p^m$ (resp.\ $F=F_p^m$) in the twisted (resp.\ non twisted) case. The group $O(\tg)$ acts naturally on the group $\tg$. We let the group $O(\tg)$ act on the group $\ofq^{\times}$ by letting $F_p$ act as $x\mapsto x^p$ and $\gamma_0$ act as $x\mapsto x^{-1}$. Again note that if $q=p^m$ then $F$ and $\gamma_0F_p^m$ (resp.\ $F_p^m$) have the same action on the group $\ofq^{\times}$ in the twisted (resp.\ non twisted) case. We denote by $Z(\tg)$ the center of $\tg$ and we let an element in $Z(\tg)$ act on the group $\ofq^{\times}$ by multiplication by its eigenvalue. We thus have an action of the group $\ZO$ on the group $\ofq^{\times}$.

Let $t\in\tg$ be a semisimple element. For any $zF_0\in\ZO$ and $\fxi\in\spec(t)/F$ we set
\begin{displaymath}
  zF_0(\fxi):=\ll F \rr.zF_0(\xi),
\end{displaymath}
and still call $zF_0$ the induced bijection $\mf_{t}\rightarrow \mf_{zF_0(t)}$, \textit{i.e.}, 
\begin{displaymath}
  zF_0(\lambda)(\fxi):=\lambda(\ll F \rr.(zF_0)^{-1}(\xi))
\end{displaymath}
for any $\fxi\in\spec(zF_0(t))/F$ and $\lambda\in\mf_{t}$.

The proof of \cite[Proposition 1]{Bru09} shows that for any $F_0\in O(\tg)$ and any semisimple element $t\in\tg$ we have:
\begin{equation}\label{eq:bru}
  \widehat{F_0(t)}=\hat{t}\circ F_0^{-1}.
\end{equation}
As the group $Z(\tg)$ is isomorphic via $\hat{}$ to the group of linear characters of $\tg$ that are trivial on $G$ (see \cite[Equation 8.19]{CabEn}), Equation \ref{eq:bru} in the particular case $t\in Z(\tg)$ shows that the group $\ZO$ acts also on the sets $\ILG$ and $\IG$ (see \S \ref{sec:2.2} and Equation \ref{eq:actionofaut}). 

\begin{prop}\label{prop:equivariancechar}
    Let $t$ be a semisimple element in $\tg$, let $\lambda\in\mf_{t}$ and let $zF_0\in\ZO$. Then
    \begin{displaymath}
       \hat{z}F_0(\tc^{\tg}_{t,\lambda})=\tc^{\tg}_{zF_0(t), zF_0(\lambda)}.
    \end{displaymath}
In particular if $zF_0(t)$ is conjugate to $t$, then $\hat{z}F_0(\tc^{\tg}_{t,\lambda})=\tc^{\tg}_{t,\lambda}$ if and only if $\lambda$ is constant on the $zF_0$-orbits of $\spec(t)/F$.
\end{prop}
\begin{proof}
We compute $\hat{z}F_0(\tc)$ with $\tc:=\tc^{\tg}_{t, \lambda}=\epsilon_{\tbg}\epsilon_{\tbg_{t}}(R_{\tbg_{t}}^{\tbg}(\widehat{t}\tvf_{\lambda}))$ (see the end of \S\ref{sec:JordanDec}):
    \begin{displaymath}
    \begin{aligned}
        \hat{z}F_0(\tc) & = \epsilon_{\tbg}\epsilon_{\tbg_{t}}\hat{z}F_0(R_{\tbg_{t}}^{\tbg}(\widehat{t}\tvf_{\lambda})) \\
        & = \epsilon_{\tbg}\epsilon_{\tbg_{t}}\hat{z}R_{\tbg_{F_0(t)}}^{\tbg}(F_0(\widehat{t})F_0(\tvf_{\lambda})) & (\text{by  \cite[Proposition 13.22]{DM}})\\
        & = \epsilon_{\tbg}\epsilon_{\tbg_{t}}\hat{z}R_{\tbg_{F_0(t)}}^{\tbg}(\widehat{F_0(t)}F_0(\tvf_{\lambda})) & (\text{by Equation \ref{eq:bru}})\\
        & = \epsilon_{\tbg}\epsilon_{\tbg_{t}}R_{\tbg_{F_0(t)}}^{\tbg}(\hat{z}\widehat{F_0(t)}F_0(\tvf_{\lambda})) & (\text{by \cite[Equation 8.20]{CabEn}}).\\
    \end{aligned}
\end{displaymath}

We now need to identify the map in $\mf_{zF_0(t)}$ associated with $F_0(\tvf_{\lambda})$ (see Definition \ref{def:mos1}). We first identify the element in $\Irr(W_{zF_0(t)})^F$ associated with $F_0(\tvf_{\lambda})$. Note that $F_0$ induces a bijection $\mathbf{W}_t\rightarrow \mathbf{W}_{zF_0(t)}$. We now compute starting from Equation \ref{eq:unipotentgt} to which we apply $F_0$:
\begin{displaymath}
    \begin{aligned}
            {F_0}(\tvf_{\lambda}) &= \epsilon_{{\rho_{\lambda}}}\frac{1}{|\mathbf{W}_{t}|}\sum_{\sigma\in \mathbf{W}_{t}}\tilde{\rho}_{\lambda}(\sigma.F)F_0(R_{{\mathbf{T}}_{{\sigma}}}^{\tbg_t}(1))
\\
&=\epsilon_{{\rho_{\lambda}}}\frac{1}{|\mathbf{W}_{t}|}\sum_{\sigma\in \mathbf{W}_{t}}\tilde{\rho}_{\lambda}(\sigma.F)R_{F_0({\mathbf{T}}_{{\sigma}})}^{F_0(\tbg_t)}(1) & (\text{by \cite[Proposition 13.22]{DM}}) \\
&=\epsilon_{{\rho_{\lambda}}}\frac{1}{|\mathbf{W}_{t}|}\sum_{\sigma\in \mathbf{W}_{t}}\tilde{\rho}_{\lambda}(\sigma.F)R_{{\mathbf{T}}_{{F_0(\sigma)}}}^{\tbg_{F_0(t)}}(1) & (\text{because ${F_0}(\mathbf{T})=\mathbf{T}$})\\
&=\epsilon_{{\rho_{\lambda}}}\frac{1}{|\mathbf{W}_{z{F_0(\mathbf{t})}}|}\sum_{\sigma\in \mathbf{W}_{zF_0(t)}}{\tilde{\rho}_{\lambda}}(F_0^{-1}(\sigma).F)R_{{\mathbf{T}}_{{\sigma}}}^{\tbg_{F_0(t)}}(1).
    \end{aligned}
\end{displaymath}
As the character $F_0(\tilde{\rho}_{\lambda})$ defined by $\sigma.F^i \mapsto \tilde{\rho}_{\lambda}(F_0^{-1}(\sigma).F^i)$ is an extension of $F_0(\rho_{\lambda})\in\Irr(\mathbf{W}_{zF_0(t)})^{F}$ to $\mathbf{W}_{zF_0(t)}\rtimes \ll F \rr$  which is realisable over the rational field, this last equation shows that $F_0(\rho_{\lambda})$ is the character corresponding to ${F_0}({\tvf}_{\lambda})\in\me(\tg_{F_0(t)},1)$. We now identify the map in $\mf_{zF_0(t)}$ associated with $F_0(\rho_{\lambda})$. But the isomorphism induced by $F_0$ from 
\begin{displaymath}\mathbf{W}_{t}\cong\prod_{\xi\in\spec(t)}\mathfrak{S}_{m_{\xi}(t)}
\end{displaymath}
 to 
\begin{displaymath}\mathbf{W}_{zF_0(t)}\cong\prod_{\xi\in\spec(zF_0(t))}\mathfrak{S}_{m_{\xi}(zF_0(t))}
\end{displaymath} 
is such that for $\xi\in\spec(t)$,
\begin{displaymath}
    F_0(\mathfrak{S}_{m_{\xi}(t)})=\mathfrak{S}_{m_{zF_0(\xi)}(zF_0(t))},
\end{displaymath}
hence the character $F_0(\rho_{\lambda})$ is associated with the map $zF_0(\lambda)$. If $zF_0(t)$ is conjugate to $t$ then arguments analogous as in the previous computation show that there exists $\tau\in\mathbf{W}$ such that the map $\tau F_0$ induces an automorphism of the group $\mathbf{W}_{t}$ satisfying $\tau F_0(\mathfrak{S}_{m_{\xi}(t)})=\mathfrak{S}_{m_{zF_0(\xi)}(t)}$ and $\hat{z}F_0(\tc)=\tc^{\tg}_{t,\tau zF_0(\lambda)}$, hence the result.
\end{proof}

\subsection{Generalised Gelfand Graev characters and unipotent support}
\label{sec:GGGR}
 We now briefly recall the basic facts from Kawanaka's theory of Generalised Gelfand Graev representations and its link with Lusztig's unipotent support. This theory is what will allow us to extend the methods developed in \cite{KT} to the unitary case.

\subsubsection*{Kawanaka's theorem}
\label{ref:propofsupport}
Let $t\in\tg$ be a semisimple element.  
\begin{definition}\label{def:leqt}
  Let $\lambda,\mu\in\mf_t$. We say that $\mu \leq_t \lambda$ if 
  \begin{displaymath}
    \text{for all }\fxi\in\spec(t)/F, \quad \mu(\fxi)\triangleleft \lambda(\fxi),
  \end{displaymath}
where $\triangleleft$ denotes the dominance order on partitions (see for instance \S $2.6$ in \cite{Geck03}).
\end{definition}

If $\mathcal{O}\in \Uni(\tbg_t)^F$, we denote by $\overline{\mathcal{O}}$ its Zariski closure in the algebraic group $\tbg_t$. Then the partial order on $\Uni(\tbg_t)^F$ given by 
\begin{displaymath}
\mathcal{O}'\leq \mathcal{O} \Leftrightarrow \mathcal{O}' \subset \overline{\mathcal{O}}
\end{displaymath}
corresponds the order defined in Definition \ref{def:leqt} under the natural bijection $\mf_t\leftrightarrow\Uni(\tbg_t)^F$ (see above Remark \ref{rmk:springercorr}). 

Kawanaka's methods allows us to associate to each element $\lambda\in\mf_t$ a character $\tilde{\Gamma}_{\lambda}$ of $\tg_t$ induced from a certain linear character of a unipotent subgroup of $\tg_t$, called a Generalised Gelfand Graev Character, also refered to as a Generalized Gelfand Graev Representation or GGGR for short. We have the following theorem whose proof is contained in \cite[Theorem 3.2.11, Corollary 3.2.18 and Remark 3.2.24(i)]{Kaw85} (see also \cite[\S 4]{Geck} 
and \cite[Corollary 13.6 and Lemma 6.3]{Tay14}):
\begin{theorem}\label{thm:wfs}
Let $\tc\in \Irr(\tg_t)$. There exists a unique element $\mos(\tc)\in\mf_t$ satisfying:
\begin{itemize}
  \item $\langle \tc,\tilde{\Gamma}_{\mos(\tc)} \rangle = 1$, and
  \item for any $\mu\in\mf_t$ such that $\langle \tc, \tilde{\Gamma}_{\mu}\rangle \neq 0$, one has $\mu\leq_t \mos(\tc)$.
\end{itemize}
Restricted to the subset $\me(\tg_t,1)\subset\Irr(\tg_t)$ of unipotent characters, the map $\tc \mapsto\mos(\tc)$ induces a bijection with the set $\mf_t$.
\end{theorem}

\begin{rmk}
Let $\mathcal{O}({\tilde{\chi}})$ denote Lusztig's unipotent support of $\tilde{\chi}\in\IG$ (see \cite{Lus92}), and we identify $\mos(\tc)$ with its associated element in $\Uni(\tbg)^F$. Let $\tilde{\chi}^*\in \Irr(\tilde{G})$ denote the Alvis-Curtis dual of $\tilde{\chi}$ up to a sign (see \cite[\S8.15]{DM}). Then $\mathcal{O}^*({\tilde{\chi}})=\mathcal{O}({\tilde{\chi}^*})$ (see \cite[Theorem 14.10]{Tay14}). 
\end{rmk}

\subsubsection*{Computation of $\mos(\tilde{\chi})$}
\label{sec:computation}

Let $t\in\tg$ be a semisimple element, and let $s$ (resp.\ $u$) be the $\ell'$-element (resp.\ $\ell$-element) such that $t=su=us$. The Jordan decomposition of characters for $\tg_s$ induces a bijection between the set $\mf_t$ and the set $\me(\tg_s,u)$ that we denote by $\lambda\mapsto\tc^{\tg_s}_{u,\lambda}$. We have $\tc^{\tg_s}_{u,\lambda}=\epsilon_{\tbg_s}\epsilon_{\tbg_t}R_{\tbg_{t}}^{\tbg_s}(\hat{u}\tvf_{\lambda})$ (see the end of \S\ref{sec:JordanDec}).

\begin{definition}\label{def:mos}
    We denote by $\me^s$ the union of all the Lusztig series $\me(\tg,sv)$ such that $v\in \tg_s$ is an $\ell$-element. We also define $\mos_{s} : \me^s \rightarrow \mathcal{F}_s$, by $\tc^{\tg}_{sv, \lambda} \mapsto \mos(\tc^{\tg_s}_{v, \lambda})$. Note that this map, restricted to the set $\me(\tg, s)$, is bijective.
\end{definition}

\begin{definition}\label{def:sumpartition}
Let $m_1, m_2$ be two positive integers and let $\lambda:=(\lambda_1\geqslant \lambda_2 \geqslant \dots \geqslant \lambda_r)$ be a partition of $m_1$ and $\mu:=(\mu_1 \geqslant \mu_2 \geqslant \dots \geqslant \mu_s )$ be a partition of $m_2$. Up to adding zeros to one of the two partitions we can suppose that $s=r$. Then we define the \emph{sum} $\lambda+\mu$ of the two partitions to be the partition $(\lambda_1+\mu_1 \geqslant \lambda_2+\mu_2 \geqslant \dots \geqslant \lambda_r+\mu_s)$ of $m_1+m_2$, whose parts are the sums of the two corresponding parts of $\lambda$ and $\mu$. This notion extends to finite sums and allows us to define a multiple of $\lambda$ denoted by $d.\lambda$ as $\lambda$ added $d$ times for $d$ a positive integer.
\end{definition}

The map $\mathcal{O}^*$ is described in \cite[\S 3]{Kaw85}. However we will follow the more explicit references \cite[\S 13.3]{Lus84} and \cite{Lus92}  (see also \cite[\S 3]{Geck}) to describe $\mos_s(\tc^{\tg}_{t, \delta})$ for $\delta\in\mf_t$. The description goes as follows:
\begin{itemize}
  \item compute $\rho'$ the $j$-induced character $j_{\mathbf{W}_{t}}^{\mathbf{W}_{s}}(\rho_{\delta}\otimes \epsilon)$ (see \cite[pp 76--77]{Lus84}, note that $\rho'$ is an $F$-stable irreducible character of $\mathbf{W}_s$), where $\epsilon$ is the sign character,
  \item via Springer's correspondence (see \cite{LS85}), $\rho'$ corresponds to a unique element in $\Uni(\tg_s)^F$. This class is labelled by an element in $\mathcal{F}_s$, which is $\mos_s(\tc)$.
\end{itemize}

\begin{definition}\label{def:uxi}
For $s$ a semisimple $\ell'$-element in $\tg$, $v\in\tg_s$ an $\ell$-element and $\xi\in\spec(s)$, let $v_{\xi}$ be the endomorphism of the $\xi$-eigenspace of $s$ induced by $v$. Then $\spec(sv)=\{\xi \w \mid \xi\in\spec(s), \w\in\spec(v_{\xi})\}$. For $\xi\in\spec(s)$, let $F_{\xi}:=F^{|\ll F \rr. \xi|}$ where $|\ll F \rr. \xi|$ is the cardinality of the $F$-orbit of $\xi$.
\end{definition}
We will abuse notations by identifying the set $\spec(u_{\xi})/F_{\xi}$ with the set $\spec(u_{F(\xi)})/F_{\xi}$ for $\xi\in\spec(s)$. Then we identify the set $\spec(su)/F$ of $F$-orbits of $\spec(su)$ with the set 
\begin{displaymath}
\{(\fxi , \ll F_{\xi}\rr.\w) \mid \fxi\in \spec(s)/F, \ll F_{\xi}\rr.\w\in\spec(u_{\xi})/F_{\xi}\}  
\end{displaymath}
via the map $\fxi\w \mapsto (\fxi ,\ll F_{\xi}\rr.\w)$. Also note that for $\xi\in\spec(s)$ and $\w\in\spec(u_{\xi})$ we have $$\sum_{\langle F_{\xi}\rangle.\w\in\spec(u_{\xi})/F_{\xi}}|\ll F_{\xi}\rr.\w|m_{\xi\w}(su)=m_{\xi}(s).$$

We now describe the $\mos_s(\tc^{\tg}_{t,\delta})$ combinatorially:
\begin{itemize}
  \item by \cite[3.1(b)]{Lus09}, $\rho'=\rho_{\lambda}\otimes\epsilon$ for $\lambda\in\mf_s$ satisfying for all $\fxi\in\spec(s)/F$,  
  \begin{equation}\label{eq:wfs}    \lambda(\ll F \rr. \xi)=\sum_{\langle F_{\xi} \rangle .\w\in\spec(u_{\xi})/F_{\xi}}{|\ll F_{\xi}\rr .\w|.\delta(\ll F \rr. \xi , \ll F_{\xi}\rr .\w)},
  \end{equation}
  \item and  by  Remark \ref{rmk:springercorr} we have $\lambda=\mos_s(\tc)$.
\end{itemize}

Note that in the case $u=1$ we have $\mos(\tc^{\tg}_{s,\delta})=\delta$.

\begin{rmk} 
Equation \ref{eq:wfs} is analogous to the formula in \cite[p.\ 481]{KT} for passing from a symbol $\mathfrak{s}$ to a symbol $\mathfrak{s}^*$. Note that our formula does not involve conjugating partitions here thanks to our parameterisation of the unipotent characters by multipartitions being the one of \cite{KT} twisted by conjugation. Otherwise the procedure we just described on multipartitions would have given Lusztig's unipotent support of the character $\tc^{\tg_s}_{u, \lambda}$. 
\end{rmk}

\subsection{Stabilisers of characters}

We now find a condition to ensure that the second condition of $O(\tg)$-replaceability is satisfied (see Definition \ref{def:Oexchange}).
\begin{definition}
\label{def:Stabaction}
If $H$ is a subgroup of $\ZO$ and $t\in\tg$ (resp.\ $\tc\in\IG$), then we denote by $H_t$ (resp.\ $H_{\tc}$) the stabiliser of the $\tg$-conjugacy class of $t$ (resp.\ of the character $\tc$) in $H$.
\end{definition}
Let $s\in\tg$ be a semisimple $\ell'$-element.
\begin{lem}\label{cor:stabsuglobal}\label{lem:stabsulocal}
    Let $u\in (\tg_s)_{\ell}$. Suppose that for all $\xi\in\spec(s)$ and all $zF_0\in (\ZlpO)_s$, $F_0(\spec(u_{\xi}))=\spec(u_{zF_0(\xi)})$. Then $(\ZlpO)_{s}=(\ZlpO)_{su}$.
\end{lem}
\begin{proof}
    We just need to prove that $(\ZlpO)_s\subseteq (\ZlpO)_{su}$. Let $zF_0\in (\ZlpO)_{s}$. It is enough to show that $zF_0(\spec(su))\subset\spec(su)$. Let $\xi\w\in\spec(su)$ with $\xi\in\spec(s)$ and $\w\in\spec(u_{\xi})$. Then $zF_0(\xi\w)=zF_0(\xi)F_0(\w)$. By hypothesis $zF_0(\xi)\in\spec(s)$ and $F_0(\w)\in\spec(u_{zF_0(\xi)})$.  Hence $zF_0(\xi\w)\in\spec(su)$.
\end{proof}

We are now able to prove the following proposition and corollary, which will be essential in proving Theorem \ref{thm:A}.
\begin{prop}\label{prop:stabtclocal}\label{cor:stabtcglobal}
Let us keep the hypothesis and setup of Lemma \ref{lem:stabsulocal}. Let $\tc=\tc^{\tg}_{s, \lambda}$ and $\tc'=\tc^{\tg}_{su,\delta}$ for $\lambda\in\mf_s$ and $\delta\in\mf_{su}$. Suppose that:
    \begin{enumerate}
      \item $\mos_s(\tc)=\mos_s(\tc')$ (see Definition \ref{def:mos}),
      \item for all $\ll F \rr. \xi\in\spec(s)/F$ the map $\delta(\ll F \rr. \xi, ?)$ is constant on $\spec(u_{\xi})/F_{\xi}$ (see Definition \ref{def:uxi}) and we still denote by $\delta$ the induced map on $\spec(s)/F$, and
      \item for all $\xi\in\spec(s)$, all the orbits of $\spec(u_{\xi})$ under the action of $F_{\xi}$ have the same cardinality.
    \end{enumerate}
Then $(\ZlpO)_{\tc}=(\ZlpO)_{\tc'}$.
\end{prop}
\begin{proof}
    By Equation \ref{eq:wfs} and our hypothesis, for all $\ll F \rr. \xi\in\spec(s)/F$, 
    \begin{displaymath}
\lambda(\ll F \rr. \xi)=\mos_s(\tc)(\ll F \rr. \xi)=\mos_s(\tc')(\ll F \rr. \xi)=|\spec(u_{\xi})|\delta(\ll F \rr. \xi).      
    \end{displaymath}
 This equation together with the hypothesis of Lemma \ref{lem:stabsulocal} proves that $\lambda$ is constant on the $zF_0$-orbits of $\spec(s)/F$ if and only if $\delta$ is. Now Lemma \ref{lem:stabsulocal} and Proposition \ref{prop:equivariancechar} imply the result.
\end{proof}

If $d$ is a positive integer
then $d_{\ell}$
 (resp. $d_{\ell'}$) denotes the unique power of $\ell$
 (resp. positive integer coprime to $\ell$) such that $d=d_{\ell}d_{\ell'}$.

\begin{cor}\label{cor:stabtcl}
    Under the hypotheses and setup of Proposition \ref{cor:stabtcglobal}, we have:
    \begin{displaymath}
      \kappa_G(\tc')_{\ell}=|(Z(\tg)_{\ell})_{su}|
    \end{displaymath}
\end{cor}
\begin{proof}
By standard Clifford theory we have $\kappa_G(\tc')_{\ell}=|(Z(\tg)_{\ell})_{\tc'}|$. As $(Z(\tg)_{\ell})_{\tc'}\subseteq (Z(\tg)_{\ell})_{su}$ (see Proposition \ref{prop:equivariancechar})  we just need to prove that $(Z(\tg)_{\ell})_{su}\subset (Z(\tg)_{\ell})_{\tc'}$. Let $z\in (Z(\tg)_{\ell})_{su}$. By Proposition \ref{prop:equivariancechar}, $\hat{z}\tc'=\tc^{\tg}_{su,\gamma}$ with $\gamma\in\mf_{su}$ such that for $\fxi\in\spec(s)/F$ and $\ll F_{\xi}\rr.\w\in\spec(u_{\xi})/F_{\xi}$ we have
\begin{displaymath}
\gamma(\ll F \rr.\xi , \ll F_{\xi} \rr. \w)=\delta(\ll F\rr .\xi ,\ll F_{\xi}\rr .z^{-1}\w).
\end{displaymath}
By hypothesis $2$ in Proposition \ref{prop:stabtclocal} we have
\begin{displaymath}
\delta(\ll F\rr .\xi , \ll F_{\xi}\rr .z^{-1}\w)=\delta(\ll F \rr. \xi , \ll F_{\xi} \rr. \w),  
\end{displaymath} 
hence $\delta=\gamma$.
\end{proof}

\section{Unitriangularity of the decomposition matrices of special linear and unitary groups}
In this section we first recall the main result in \cite{Geck} and then show that Theorem \ref{thm:equiv} is applicable. We finish by proving that it implies Theorem \ref{thm:A}.

\subsection{Unitriangularity of the decomposition matrices of general linear and unitary groups}

For any semisimple $\ell'$-element $s$ in $\tg$, the partial order $\leq_s$
(see Definition \ref{def:leqt}) on $\mf_{s}$ induces a partial pre-order
on $\me^s$ via the map $\mos_s$ (see Definition \ref{def:mos}) which induces an actual partial order on the set $\me(\tg,s)$, and also induces an actual partial order on $\me(\tg_s,1)$ (see Definition \ref{def:mos1}). The notation $\leq_s$ will be used to denote
the induced order on any of these sets. Let $\leq$ be the
partial pre-order on $\Irr(\tg)$ such that $\tc\leq\tc'$ if and only if $\tc,\tc'\in \me^s$ and $\tc\leq_s \tc'$ for some semisimple $\ell'$-element $s\in\tg$.

Let $\tme=\cup_{s}\me(\tg,s)$ (union over semisimple $\ell'$ elements). Note that by \cite[Theorem A]{GeckII} we have $\tme_{\ell}=\IBr(\tg)$ (see Definition \ref{def:mel}). Moreover in \cite{Geck} the following is proved:
\begin{theorem}\label{thmMeinolf}
 For any semisimple $\ell'$-element $s\in \tilde{G}$, there exists an
 injection $\Theta_s : (\mathcal{E}^s)_{\ell} \rightarrow \mathcal{E}^s$
 (see Definition \ref{def:mel}) with image $\mathcal{E}(\tilde{G},s)$ such that for $\tilde{\eta},
 \tilde{\mu}\in (\mathcal{E}^s)_{\ell}$, if we define $\tth=\cup_{s}\Theta_s$ and set
\begin{displaymath}
  \tilde{\eta}\leq \tilde{\mu} \Leftrightarrow \Theta(\tilde{\eta}) \leq \Theta(\tilde{\mu}),
\end{displaymath}
then the map $\tth$ has image $\tme$, and together with the partial order $\leq$ make
the decomposition matrix of $\tg$ unitriangular. 
\end{theorem}

Proposition \ref{prop:equivariancechar} shows that $\tme$ is $\ZlpO$-stable (see Definition \ref{def:OG}) so that, by Lemma \ref{lem:equivar}, the map $\tth$ is $\ZlpO$-equivariant.

The next section will prove that every element of $\tme$ is
$O(\tg)$-replaceable (see Definition \ref{def:Oexchange}), thus allowing us to use Theorem \ref{thm:equiv}. First we show that regarding the $\kappa_G$-replaceability (see Definition \ref{def:exchange}) only the
$\ell$-part of the number of irreducible constituents is relevant. 
\begin{cor}\label{lprimepart}
  Let $s$ be a semisimple $\ell'$-element of $\tilde{G}$. Let
  $\tilde{\chi}\in 
\tme$ and let $\tilde{\eta}\in\ILG$
  be such that $\tth(\tilde{\eta})=\tilde{\chi}$. Let
  $\tc'\in\Irr(\tg)$ and suppose that $\tc$ is replaceable by $\tc'$ and have the
  same stabiliser under the action of the group $Z(\tg)_{\ell'}$. 
Then $\kappa_G(\tilde{\chi}')_{\ell'}= \kappa_{G}(\tilde{\eta})_{\ell'}$.
\end{cor}
\begin{proof}
By Clifford theory we have $\kappa_G(\tc)_{\ell'}=\kappa_G(\tc')_{\ell'}=|(Z(\tg)_{\ell'})_{\tc}|$. By the equivariance of the map $\tth$ we have $(Z(\tg)_{\ell'})_{\te} = (Z(\tg)_{\ell'})_{\tc}$.

Applying Clifford theory (see \cite[Proposition 3.2(i) and Lemma 3.1]{KT}) we have $|(Z(\tg)_{\ell'})_{\te}|=\kappa_G(\te)_{\ell'}$, which proves the lemma.
\end{proof}

\subsection{Changing the basic set while preserving unitriangularity}

We now prove a criterion regarding the replaceability of an element of $\tme$.

\begin{prop}\label{lembschange1}
  Let $s$ be a semisimple $\ell'$-element in $\tilde{G}$.
  Let $\tilde{\chi}\in\mathcal{E}(\tilde{G}, s)$. Let $u\in\tg_s$ be an
  $\ell$-element and let $\tilde{\chi}'\in
  \mathcal{E}(\tilde{G},su)$ be such that (see Definition \ref{def:mos})
\begin{displaymath}
\mathcal{O}^*_{s}(\tilde{\chi})=\mathcal{O}^*_{s}(\tilde{\chi}').\end{displaymath}
Then $\tc$ is replaceable by $\tc'$ (see Definition \ref{def:exchange}).
\end{prop}
\begin{proof}
 Write $\tc=\tc^{\tg}_{s,   \lambda}$ and $\tc'=\tc^{\tg}_{su, \delta}$ for some
$\lambda\in\mf_s$ and $\delta\in\mf_{su}$. Let us denote by
$\tilde{\chi}^{\tg_s}_{\delta}$ the class function $\epsilon_{\tbg_{su}}\epsilon_{\tbg_s}R_{\tbg_{su}}^{\tbg_s}(\tvf_{\delta})\in\mathbb{Z}\Irr(\tg_s)$ (see Definition \ref{def:mos1}).

By \cite[Proposition 4.11]{Br}, we have:
\begin{displaymath}
  d^1(\tilde{\chi}')=\sum_{\mu\in \mf_s}\langle \tvf_{\mu},\tilde{\chi}^{\tg_s}_{\delta} \rangle d^1(\tilde{\chi}^{\tilde{G}}_{s, \mu}),
\end{displaymath}
where $d^1$ is the decomposition map (see \S\ref{sec:unitrigeneral}). So, using Theorem \hyperref[thmMeinolf]{\ref{thmMeinolf}} and Equation
\ref{eq:red}, we will be done once we show the following:
\begin{equation}\label{eq:suff}
\langle \tvf_{\lambda},\tilde{\chi}^{\tg_s}_{\delta} \rangle=1, \text{ and }
\mu \nleq_{s} \lambda \Rightarrow \langle \tvf_{\mu},\tilde{\chi}^{\tg_s}_{\delta} \rangle=0 \textrm{ (see Definition \ref{def:leqt})}.
\end{equation} 
We fix for the end of this proof a linear order refining $\leq_s$. Recall that for $\mu'\in\mf_s$ we denoted by $\tilde{\Gamma}_{\mu'}$ the corresponding GGGR of $\tg_s$ (see \S \ref{sec:GGGR}). By Theorem \ref{thm:wfs}, the square matrix $D=(\langle \tvf_{\mu} ,\tilde{\Gamma}_{\mu'}
\rangle)_{\mu, \mu' \in \mf_s}$ ordered according to the fixed linear order just set, is lower unitriangular. Consider now the row matrices
$L_1=(\langle \tilde{\chi}^{\tg_s}_{
      \delta}, \tvf_{\mu}\rangle)_{\mu\in\mf_s}$ and 
$L_2=(\langle \tilde{\chi}^{\tg_s}_{
      \delta},\tilde{\Gamma}_{\mu'}\rangle)_{\mu'\in\mf_s}$. 

By \cite[Proposition 8.25]{CabEn} the class function
$\tilde{\chi}^{\tg_s}_{\delta}$ decomposes in $\mathbb{Z}\me(\tg_s,1)$ so we have the matrix equality $L_1\times D=L_2$. Let $\mu'\in\mf_s$. As a GGGR is zero on non unipotent elements (recall that it is induced from a unipotent subgroup) we have that
$\langle \tilde{\chi}^{\tg_s}_{u, \delta},\tilde{\Gamma}_{\mu'}
\rangle=\langle d^1
\tilde{\chi}^{\tg_s}_{u, \delta},\tilde{\Gamma}_{\mu'}
\rangle$. Moreover as $u$ is an $\ell$-element, hence $\hat{u}$ is
also one, we have that
$d^1(\hat{u}\tvf_{\delta})=d^1(\tvf_{\delta})$. This
combined with the fact that the map $d^1$ commutes with the Deligne-Lusztig
induction map by \cite[Theorem 12.6(i)]{DM} yields 
\begin{equation}\label{eq:egal}
\forall \mu'\in\mf_s, \quad \langle \tilde{\chi}^{\tg_s}_{\delta},\tilde{\Gamma}_{\mu'}
\rangle=\langle \tilde{\chi}^{\tg_s}_{u, \delta},\tilde{\Gamma}_{\mu'} \rangle. 
\end{equation}

Now our hypothesis and Theorem \ref{thm:wfs} imply:
\begin{displaymath}
  \forall \mu'\in\mf_s, \mu' \nleq_s \lambda \Rightarrow \langle
 \tilde{\chi}^{\tg_s}_{u, \delta},\tilde{\Gamma}_{\mu'} \rangle=0
\end{displaymath}
and
\begin{displaymath}
\langle\tilde{\chi}^{\tg_s}_{u, \delta},\tilde{\Gamma}_{\lambda} \rangle=1.
\end{displaymath}
This, together with Equation \ref{eq:egal} and the matrix equality
$L_1 = D^{-1}L_2$ implies Equation \ref{eq:suff}. Hence the result.
 \end{proof}

\subsection{Finding suitable $\ell$-elements in centralisers of semisimple  $\ell'$-elements}
We will now prove that for all semisimple $\ell'$-element $s\in\tg$,
every character in $\me(\tg,s)$ is $O(\tg)$-replaceable (see Definition \ref{def:Oexchange}). We first deal
with the second condition about stabilisers in Propositions \ref{prop:lelement1} and \ref{lemineq}.

\begin{prop}\label{prop:lelement1}
  Let $s$ be a semisimple $\ell'$-element in $\tilde{G}$. Let $a$ be an integer such that for all $\xi\in \spec(s)$,
  $\ell^a$ divides $\gcd(m_{\xi}(s), |Z(\tilde{G})|)$ (see \S \ref{sec:JordanDec}). Let $z\in
  Z(\tg)$ be an element of order $\ell^a$ and let $\w$ be its eigenvalue. Then there exists an
  $\ell$-element $u\in \tg_s$ depending up to conjugacy only on
  the conjugacy class of $s$ and on the number $\ell^a$ such that:
  \begin{enumerate}
      \item   for all $\xi\in\spec(s)$, $\spec(u_{\xi})=\{\w^i, 0\leq i\leq \ell^a-1\}$ (see
  Definition \ref{def:uxi}) and is thus independent of $\xi$, and
      \item   $|(Z(\tg)_{su})_{\ell}|=\ell^a$ (see Definition \ref{def:Stabaction}).
\end{enumerate}
\end{prop}
\begin{proof}
Let $\epsilon\in\{\pm 1\}$ be such that $\tg=\GL_n(\epsilon q)$. Recall that $\mathbf{T}$ denotes the maximal diagonal torus of $\tbg$,  let $\mathbf{T}'=\leftexp{g}{\mathbf{T}}$ for $g\in\tbg$ be an $F$-stable maximal torus containing $s$ and let $w\in\mathbf{W}$ be the canonical image of $g^{-1}F(g)$. Recall that (see Equation \ref{eq:actiononG})
\begin{displaymath}
\leftexp{g^{-1}}{\tg_s}=\tbg_{\mathbf{s}}^{wF}=\prod_{\fxi\in\spec(s)/F}{\tbg^{wF}_{\fxi}}
\end{displaymath}
where $\leftexp{g}{\mathbf{s}}=s$ and $\GL_{m_{\xi}(s)}((\epsilon q)^{|\fxi|})\cong\tbg_{\fxi}^{wF}$ through the map $$M\mapsto (M,F(M),\dots, F^{|\fxi|-1}(M))$$ for $M\in\GL_{m_{\xi}(s)}((\epsilon q)^{|\fxi|})$. For $\xi\in\spec(s)$ we set $u_{\xi}\in\tg_s$ such that its image in $\GL_{m_{\xi}(s)}((\epsilon q)^{|\fxi|})$ is a diagonal matrix with entries in the set $\{\w^i, 0\leq i\leq\ell^a-1\}$, each element in this set being repeated exactly $\frac{m_{\xi}(s)}{\ell^a}$ times.  Clearly this defines an $\ell$-element satisfying the first point. Let $x\in Z(\tg)_{\ell}$ and $\tau$ be its
eigenvalue. Suppose that $xsu$ is conjugate to $su$. Then arguing on eigenvalues we see that
there exists $i$ such that $\tau \w=\w^i$. Hence $x=z^{i-1}$. Moreover
$z\in (Z(\tg)_{su})_{\ell}$. Hence the result.
\end{proof}

\begin{rmk}
In \cite[Lemma 6.1 and Theorem 6.3]{KT} and thus in the case where $\tg=\GL_n(q)$, the conjugacy class of a given element $u$ is described by its elementary divisors as follows. For each $\fxi\in\spec(s)/F$, choose $P_{\fxi}$ the minimal polynomial over $\F_{q^{|\fxi|}}$ of a primitive $(q^{|\ll F\rr.\xi|}-1)_{\ell}\ell^a$-th root of unity so that for any $z\in (Z(\tg)_{\ell'})_s$, $P_{\ll F\rr.\xi}=P_{\ll F\rr. z\xi}$. Then the elementary divisors of $u$ are the polynomials $P_{\fxi}$, each repeated $\frac{m_{\xi}(s)}{\ell^a}$ times.
This has an analogue when
$\tg=\GU_n(q)$ by taking $P_{\fxi}$ to be:
\begin{itemize}
\item the minimal polynomial over $\F_{q^{|\fxi|}}$ of a primitive $(q^{|\ll F\rr.\xi|}-1)_{\ell}\ell^a$-th root of unity if $|\fxi|$ is even\item the minimal polynomial over $\F_{q^{2|\fxi|}}$  of a primitive $(q^{|\ll F\rr.\xi|}+1)_{\ell}\ell^a$-th root of unity if $|\fxi|$ and $\ell$ are odd,
\item the product $p_{\fxi}\tilde{p}_{\fxi}$ where $p_{\fxi}$ is the minimal polynomial over $\F_{q^{2|\fxi|}}$  of a primitive $(q^{|\ll F\rr.\xi|}+1)_{\ell}\ell^a$-th root of unity, and $\tilde{p}_{\fxi}$ is the monic polynomial whose roots are those of $p_{\fxi}$ raised to the
  $(-q)^{|\ll F \rr. \xi|}$-th power, if $|\fxi|$ is odd and $\ell=2$.
\end{itemize}

With such an element $u$ we would be able to $Z(\tg)_{\ell'}$-replace characters between $\me(\tg,s)$
and $\me(\tg, su)$ (see \S \ref{sec:conclusion}). Note that the $\tg_s$-conjugacy class of $u$ is not unique.
\end{rmk}

\begin{definition}
\label{def:gcdO}
For $s$ a semisimple $\ell'$-element in $\tg$ and $\lambda\in\mf_s$ (see
Definition \ref{def:mft}) we denote by $\gcd(|Z(\tg)|,\lambda)$
the greatest common divisor of $|Z(\tg)|$ and all the parts of all the
partitions in the image of the map $\lambda$.
\end{definition}
\begin{prop}\label{lemineq}
  Let $s$ be a semisimple $\ell'$-element of $\tilde{G}$. Let $\tilde{\chi}\in
  \mathcal{E}(\tilde{G},s)$ and let $a$ be such that $\ell^a=\gcd(|Z(\tilde{G})|, \mathcal{O}^*_{s}(\tilde{\chi}))_{\ell}$ (see Definition \ref{def:mos}). Let
  $\tilde{\eta}$ be such that
  $\Theta_s(\tilde{\eta})=\tilde{\chi}$ (see Theorem
  \hyperref[thmMeinolf]{\ref{thmMeinolf}}). Then there exists
  $u\in\tg_s$ an $\ell$-element and $\delta\in\mf_{su}$ such
  that the character $\tc$ is replaceable by $\tc'=\tc^{\tg}_{su, \delta}$ and
  \begin{displaymath}
      (\ZlpO)_{\tc}=(\ZlpO)_{\tc'}.
  \end{displaymath}
Moreover $\kappa_{G}(\tc')=\ell^a$. In particular $\kappa_G(\te)\geq \ell^a$.
\end{prop}
\begin{proof}
Note that Proposition \ref{prop:lelement1} is applicable and let $u$ be
as in its conclusion. We have 
\begin{displaymath}
  \spec(su)/F=\{(\ll F \rr. \xi , \{\w^i\}) \mid
\ll F \rr. \xi\in\spec(s)/F,  0 \leq i \leq \ell^a-1\}
\end{displaymath}
where $\spec(u)=\{\w^i \mid 0\leq
i\leq \ell^a-1\}$. Then we define $\delta\in\mf_{su}$ such that for all $\ll F \rr. \xi\in \spec(s)/F$ and $\w^i\in\spec(u)$,  $\delta(\ll F \rr. \xi, \{\w^i\})$ is the partition of $\frac{m_{\xi}(s)}{\ell^a}$
such that 
\begin{displaymath}
  \ell^a. \delta(\ll F \rr. \xi,\{w^i\})=\mos_s(\tc)(\ll F \rr. \xi),
\end{displaymath}
and we set
\begin{displaymath} 
\tilde{\chi}':=\tilde{\chi}^{\tilde{G}}_{su,\delta}.
\end{displaymath}
Then by Equation \ref{eq:wfs} we have that
$\mos_s(\tc')=\mos_s(\tc)$. Hence, by Proposition \ref{lembschange1}, $\tc$ is replaceable by $\tc'$. We now want to apply
Proposition \ref{cor:stabtcglobal}. Then note that for all
$\xi\in\spec(s)$, $\spec(u_{\xi})$ is $O(\tg)$-stable as it consists
of a subgroup of $(\overline{\F}_q)_{\ell}^{\times}$ and the map
$\xi\mapsto\spec(u_{\xi})$ is constant on $\spec(s)$, so that
Proposition \ref{cor:stabsuglobal} is applicable. The first
hypothesis of Proposition \ref{prop:stabtclocal} is already checked while
the second follows from the fact that for all $\xi\in\spec(s)$, the
set $\spec(u_{\xi})/F_{\xi}$ is equal to $\spec(u_{\xi})=\spec(u)$ and the map
$\delta(\ll F \rr. \xi , ?)$ was defined as constant on this
set. The third hypothesis of Proposition \ref{cor:stabtcglobal} is
satisfied as for all $\xi\in\spec(s)$, the $F_{\xi}$-orbits on
$\spec(u_{\xi})$ all have cardinality $1$. Hence Proposition
\ref{cor:stabtcglobal} is applicable and yields the equality of
stabilisers. By Corollary \ref{cor:stabtcl} and Proposition
\ref{prop:lelement1} we also obtain that
$\kappa_G(\tc')_{\ell}=\ell^a$, and by Remark \ref{rmk:multiple} we have that $\kappa_G(\te)\geq\ell^a$.
\end{proof}

\begin{rmk}
Let us keep the hypothesis and notation of Proposition \ref{lemineq}
and of its proof, but replace
$u$ by an element as in the remark below Proposition
\ref{prop:lelement1}. Then the
character $\tc'=\tilde{\chi}^{\tg}_{su, \delta}$ with $\delta\in\mf_{su}$ as in the proof of Proposition \ref{lemineq} is
replaceable by $\tc$ and
the conclusion of Proposition \ref{lemineq} is true except that we
need to replace the group $\ZlpO$ by $Z(\tg)_{\ell'}$. This is what is
done in \cite[Theorem 6.3]{KT} in the case $\tg=\GL_n(q)$. Note that as the
$\tg_s$-conjugacy class of the element $u$ is not unique, there are
multiple choices of characters $\tc'\in \me^s$ with
which we can replace  $\tc$ with this method.
\end{rmk}

To apply Theorem \ref{thm:equiv} with the map $\tth$ defined in Theorem \ref{thmMeinolf}, we now need to prove that the characters $\tc$ and $\tc'$ as in Proposition \ref{lemineq} satisfy the first condition of $O(\tg)$-replaceability (see Definition \ref{def:Oexchange}). This condition involves computing $\kappa_G(\th_s^{-1}(\tc))_{\ell}$, which is the aim of this last proposition.

\begin{prop}\label{lemeq}
  Let $s$ be a semisimple $\ell'$-element of $\tilde{G}$. Then for all $\tilde{\chi}\in \mathcal{E}(\tilde{G},s)$,  $\kappa_G(\Theta^{-1}_s(\tilde{\chi}))_{\ell}= \gcd(|Z(\tg)|,
  \mathcal{O}^*_{s}(\tilde{\chi}))_{\ell}$ (see Theorem \ref{thmMeinolf} and Definition \ref{def:gcdO}) .
\end{prop}
\begin{proof}
Let $R\subseteq\tg$ be the subgroup containing $G$ such that the factor group $R/G$ is the largest subgroup of order coprime to $\ell$ in the group $\tg/G$. Note that $\gcd(|Z(\tg)|,
  \mathcal{O}^*_{s}(\tilde{\chi}))_{\ell}=\gcd((\tg:R),
  \mathcal{O}^*_{s}(\tilde{\chi}))$. Note also that if $\tilde{G}=\GL_n(q)$ then our claim is \cite[Proposition
4.7]{KT}. So we can suppose that $\tilde{G}=\GU_n(q)$, though this is not necessary: just replace $q$ by $-q$ in what follows. 

By Proposition \hyperref[lemineq]{\ref{lemineq}} we have the
inequality
\begin{equation}
    \label{eq:1}
\kappa_G(\Theta^{-1}_s(\tilde{\chi}))_{\ell}\geq \gcd((\tg:R), \mos_s(\tc)).    
\end{equation}
To show that we have equality we will apply a counting argument as in \cite[Theorem 4.7]{KT}. It involves counting the number of conjugacy classes of $\ell'$-elements in $R$. To do this we show that if $h=su$ (Jordan decomposition) is an $\ell'$-element in
$\tg$, then $h\in R$ and 
\begin{equation}\label{eq:lconj}
  \frac{|h^{\tg}|}{|h^{R}|}=\gcd((\tg:R),\mos_s(u))
\end{equation}
where $\mos_s(u)$ denotes the element in $\mf_s$ (see Definition \ref{def:mft}) labelling the
unipotent conjugacy class of $u$ in $\tg_s$ (here the
element $u$ is a unipotent element) and $h^{\tg}$ (resp.\ $h^R$)
is the $\tg$-conjugacy class (resp.\ $R$-conjugacy class) of $h$. The fact that $h\in R$ is obvious. Then \cite[Lemma 2.2]{KT} which is a general group theoretic lemma gives that $\frac{|h^{\tg}|}{|h^{R}|}=\gcd((\tg:R), c)$ where $c$ is the index of $C_{\tg}(h)G$ in $\tg$, which is the same as the index of $\det(C_{\tg}(h))$ in the group of $(q+1)$-th roots of unity in $\overline{\F}^{\times}_q$. We now show that this index is $\gcd(q+1,\mos_s(u))$  (see Definition \ref{def:gcdO}), which will prove Equation \ref{eq:lconj}. 

Let $\epsilon\in\{\pm 1\}$. Observe that, using the same arguments as in \cite[Lemma 2.3]{KT}, if $u$ is a unipotent element in $\GL_m(\epsilon q^d)$ whose Jordan normal form corresponds to the partition $\lambda:=(\lambda_1,..., \lambda_r)$ of $m$ then $\det$ maps $C_{\GL_m(\epsilon q^d)}(u)$ onto the subgroup of index $\gcd(\epsilon q^d-1,\lambda_1,...,\lambda_r)$ of the group of $(\epsilon q^d-1)$-th roots of unity in $\overline{\F}_q^{\times}$. Now let $h=su$ be an $\ell'$-element in $\tg$. Let $\mathbf{T}'=\leftexp{g}{\mathbf{T}}$ be an $F$-stable maximal torus containing $s$ and let $w\in\mathbf{W}$ be the canonical image of $g^{-1}F(g)$. Then $\det(C_{\tg}(su))=\det(C_{\tg_s}(u))=\det(C_{\tg^{wF}_{\mathbf{s}}}(\mathbf{u}))$ where $\leftexp{g}{\mathbf{s}}=s$ and $\leftexp{g}{\mathbf{u}}=u$. Recall that (see Equation \ref{eq:actiononG})
\begin{displaymath}
\tbg_{\mathbf{s}}^{wF}=\prod_{\fxi\in\spec(s)/F}{\tbg^{wF}_{\fxi}}
\end{displaymath}
where $\GL_{m_{\xi}(s)}((-q)^{|\fxi|})\cong\tbg_{\fxi}^{wF}$ through the map $$M\mapsto (M,F(M),\dots, F^{|\fxi|-1}(M))$$ for $M\in\GL_{m_{\xi}(s)}((-q)^{|\fxi|})$, and let us write $\mathbf{u}=\prod_{\fxi\in \spec(s)/F}\mathbf{u}_{\fxi}$, so that 
\begin{displaymath}
  C_{\tg^{wF}_{\mathbf{s}}}(\mathbf{u})=\prod_{\ll F\rr.\xi\in\spec(s)/F}{C_{\tbg_{\fxi}^{wF}}(\mathbf{u}_{\fxi})}.
\end{displaymath}

Each unipotent element $\mathbf{u}_{\fxi}$ has its conjugacy class labelled by a partition of $m_{\xi}(s)$, which we denote by $\mos_s(\mathbf{u}_{\fxi})$. The map $\mos_s(u)\in\mf_s$ is the map $\fxi\mapsto \mos_s(\mathbf{u}_{\fxi})$. If $\fxi\in\spec(s)/F$ and $\delta$ is a primitive $((-q)^{|\ll F\rr.\xi|}-1)$-th root of $1$, then note that $\delta^{1-q+q^2-...+(-q)^{|\ll F\rr.\xi|-1}}$ is a primitive $(q+1)$-th root of unity. We deduce that $\det$ maps $C_{\tbg_{\fxi}^{wF}}(\mathbf{u}_{\fxi})$ onto the subgroup of index $\gcd(q+1,\mos_s(\mathbf{u}_{\fxi}))$ of the group of $(q+1)$-th roots of unity. This proves that $\det$ maps $C_{\tg}(su)=\leftexp{g}{C_{\tg^{wF}_{\mathbf{s}}}(\mathbf{u})}$ onto the subgroup of index $\gcd(q+1,\mos_s(u))$ of the group of $(q+1)$-th roots of unity.

We can now apply the same counting argument as in
\cite[Theorem 4.7]{KT}, which goes as
follows. For any $\tilde{\chi}\in
\tme$ (see the line following Theorem \ref{thmMeinolf}), we let $s_{\tc}$ be a semisimple
$\ell'$-element such that $\tc\in\me(\tg,s)$. Let $\tc\in\tme$, by Clifford theory we have
$\kappa_G(\Theta^{-1}_{s_{\tc}}(\tilde{\chi}))_{\ell}=\kappa_R(\Theta^{-1}_{s_{\tc}}(\tilde{\chi}))$ (see
\cite[Lemma 3.1]{KT}). Then we combine Equation
\ref{eq:lconj}, the fact that for all semisimple $\ell'$-elements $s$ both sets $\me(\tg,s)$ and $\Uni(\tg_s)$ (the set of unipotent conjugacy classes in $\tg_s$, see above Remark \ref{rmk:springercorr})
are in bijection with $\mf_s$, and finally the fact that two Brauer characters of $\tg$ have the same restriction to $R$ if and only if they are equal,
to obtain that the equality between the number of Brauer
characters of $R$ and the number of its $\ell'$-conjugacy classes
reads:
\begin{displaymath}
\sum_{\tilde{\chi} \in \tme}{\kappa_G(\tth^{-1}(\tilde{\chi}))_{\ell}}=\sum_{\tilde{\chi}\in
  \tme}{{\gcd((\tilde{G} : R),
    \mathcal{O}^*_{s_{\tc}}(\tilde{\chi}))}}. 
\end{displaymath}
This, together with Equation \ref{eq:1}, proves  that indeed $\kappa_G(\Theta_s^{-1}(\tilde{\chi}))_{\ell}=
\gcd((\tilde{G} : R), \mathcal{O}^*_{s}(\tilde{\chi}))$ for all $\tc\in \me(\tg,s)$ and all semisimple $\ell'$-elements $s$.
\end{proof}

We are now in a position to prove Theorem \ref{thm:A} in full.

\begin{proof}[Proof of Theorem \ref{thm:A}]
  By Proposition \ref{lemineq}, Corollary \ref{lprimepart} and Proposition \ref{lemeq}, any $\tc\in\tme$ is $O(\tg)$-replaceable (see Definition \ref{def:Oexchange}). We apply Theorem \ref{thm:equiv}, which together with the description of automorphisms of $G$ (see for instance \cite[Theorem 30 p. 158 and Theorem 36 p. 195]{St68}), proves our main result.
\end{proof}

\subsection{Concluding remarks}
\label{sec:conclusion}
The remark below Proposition \ref{lemineq} together with Proposition
\ref{lemeq} shows that for all $\tc\in\tme$ (see below Theorem \ref{thmMeinolf}), $\tc$ is
$Z(\tg)_{\ell'}$-replaceable (see Definition
\ref{def:Oexchange}). Theorem \ref{thm:equiv} yields an explicit unitriangular basic set for
$G$.  Such a basic set is one that can be built using the methods described in \cite{KT} for
$\SL_n(q)$. However in general such a basic set is not stable under the action of automorphisms. For example let $\tg=\GL_3(4), G=\SL_3(4)$ and let
$\w\in\F_4$ be a generator of $\F_4^{\times}$. Let $\ell=3$, $\tth$ be as in Theorem \ref{thmMeinolf}, $\tc$ be the Steinberg character and let $\te$ be such that $\tth(\te)=\tc$. Then we have $\kappa_G(\tc)=1$ and $\kappa_G(\te)=3$ (by Proposition \ref{lemeq}), so we need to replace the Steinberg character if we are to apply Theorem \ref{thm:1} (in fact this is the only character that we have to replace). The methods in \cite{KT} suggest to exchange the
Steinberg character with the only character in $\me(\tg, u)$ with
$u=\begin{pmatrix} 0 & 0 & \w \\ 1 & 0 & 0\\ 0 & 1 & 0\end{pmatrix}$
or $u=\begin{pmatrix} 0 & 0 & \w^{2} \\ 1 & 0 & 0\\ 0 & 1
    &0 \end{pmatrix}$. Both characters have $3$ irreducible constituents upon restriction to $G$ so by Clifford theory their restriction to $G$ is not the same. Moreover they are conjugate under $O(\tg)$. We obtain by restriction two $\Out(G)$-conjugated unitriangular basic sets for $G$. The methods developed in
this paper suggest to
replace the Steinberg character with
the only character in $\me(\tg, \diag(1,\w, \w^2))$, which is an
$O(\tg)$-stable replacement for the Steinberg character. The situation is similar with $\tg=\GU_3(2)$ and $G=\SU_3(2)$.

To finish, let us stress that the unitriangular basic set obtained for
$\SL_n(q)$ and $\SU_n(q)$ is explicit. Let $\tme=\cup_{s}\me(\tg,s)$ be
the usual basic set for $\tg$. For each semisimple
$\ell'$-element $s$ in $\tg$ and each element $\lambda\in\mf_s$ we
do the following:
\begin{itemize}
  \item if $\ell^a:=\gcd(|Z(\tg)|, \lambda)_{\ell}\neq (Z(\tg)_{\ell})_{\tc^{\tg}_{s,\lambda}}$ then
  take an $\ell$-element $u\in\tg_s$ as in Proposition
  \ref{prop:lelement1},
\item let $\delta\in \mf_{su}$ be such that $\delta(\ll F \rr. \xi, \{\w^i\})=\frac{\lambda(\ll F \rr. \xi)}{\ell^a}$ for all $\xi\in \spec(s)/F$ and all $i\in \{0,1,2, \dots, \ell^a-1\}$ where $\w$ is an element in $\overline{\F}^{\times}_q$ having order $\ell^a$, and
  \item replace $\tc^{\tg}_{s, \lambda}$ by $\tc^{\tg}_{su,\delta}$ in $\tme$ to obtain a new unitriangular basic set
  that we denote by $\tme'$.
\end{itemize}
Then the unitriangular basic set $\tme'$ obtained
for $\tg$ is such that its set of irreducible constituents upon restriction to $G$ is
a unitriangular basic set for $G$ that is stable under the action of $\Out(G)$.
\bibliography{biblio2}

\end{document}